\documentclass{article}
\usepackage{amssymb,amsthm,amsmath}
\usepackage{xcolor}
\usepackage{listings}
\usepackage{caption}
\usepackage{comment}

\theoremstyle{plain}
\newtheorem{theorem}{Theorem}[section]
\newtheorem{lemma}[theorem]{Lemma}

\newtheorem{proposition}[theorem]{Proposition}
\theoremstyle{definition}

\newtheorem{remark}[theorem]{Remark}
\newtheorem{assumption}[theorem]{Assumption}

\numberwithin{equation}{section}

\DeclareCaptionFormat{listing}{\rule{\dimexpr\textwidth\relax}{1pt}\par#1#2#3}
\title{Convergence Analysis of Levenberg-Marquardt Method for Inverse Problem \\ with H\"{o}lder Stability Estimate}
\author{Akari Ishida, Sei Nagayasu and Gen Nakamura}
\date{}

\begin{document}
\maketitle
\begin{abstract}
We analyze convergence of the Levenberg-Marquardt method for solving nonlinear inverse problems in Hilbert spaces. Specifically, we establish local convergence and convergence rates for a class of inverse problems that satisfy H\"{o}lder stability estimate. Furthermore, based on what we found in the mentioned analysis, we develop global reconstruction algorithms for solving inverse problems with finite measurements for exact and noisy data, respectively. 
\end{abstract}

\section{Introduction} \label{section:introduction}
In this paper, we study nonlinear operator equation~$F(x)=y$ that arises from a nonlinear inverse problem in Hilbert spaces~$X,\,Y$, where $x\in X$ and $y\in Y$ are an exact data and a solution, respectively. The data~$y$ may include noise, which we refer to $y$ as noisy data. This problem is ill-posed in the sense that its solution $x$ could be not unique, and even if it is unique, $x$ would not depend continuously on $y$. This ill-posedness becomes more severe for noisy data. Therefore, regularization methods are required for stably deriving a solution and an approximate solution for exact and noisy data, respectively.

Several regularization methods, such as the Landweber iteration and Levenberg-Marquardt method~(LM-method), have been developed to solve nonlinear inverse problems. An overview of nonlinear regularization methods can be found in Kaltenbacher-Neubauer-Scherzer~\cite{KNS}, Kirsch~\cite{Kirsch} and Schuster-Kaltenbacher-Hofmann-Kazimierski~\cite{SKHK}, for instance. As far as we know, when many regularization methods were first proposed, their convergence was proved under the so-called tangential cone condition~(see \eqref{TCC} in Section~\ref{section:exact}). Since the algorithm of the Landweber iteration is simple, it has been extensively studied.

We first mention some results related to the Landweber iteration for nonlinear operator equations. The Landweber iteration for nonlinear operator equations in Hilbert spaces was first developed in Hanke-Neubauer-Scherzer~\cite{HNS} in 1995. Since the 2000s, active studies on regularization in Banach spaces have started. For examples, in 2006, Sch\"{o}pfer-Louis-Schuster~\cite{SLS} proposed the Landweber iteration for linear operator equations in Banach spaces. This work was extended to nonlinear operator equations by Kaltenbacher-Sch\"{o}pfer-Schuster~\cite{KSS} in 2009. We note that both \cite{HNS} and \cite{KSS} studied convergence of the Landweber iteration for nonlinear operator equations under the tangential cone condition. On the other hand, de Hoop-Qiu-Scherzer~\cite{dHQS} studied the Landweber iteration for exact data assuming H\"{o}lder stability estimate (see \eqref{F;inv-Holder} in Section~\ref{section:exact}). This result was extended to noisy data case by Mittal-Giri~\cite{MG}. We remark here that combining H\"{o}lder stability estimate with the Fr\'{e}chet differentiability of $F$ and the Lipschitz continuity of $F'$ implies the tangential cone condition (see Remark~\ref{rmk:TCC;Holder} in Section~\ref{section:exact}). In \cite{HNS} and \cite{KSS}, convergence rates were given by assuming both the tangential cone condition and an additional condition on the data, known as the source condition. In contrast, \cite{dHQS} and \cite{MG} demonstrate that H\"{o}lder stability estimate can serve as a single, alternative condition that substitutes for both the tangential cone condition and the source condition, making it possible to establish convergence rates.

The following is known regarding the stability estimate for typical inverse problems related to partial differential equations, such as electrical impedance tomography (EIT) for determining unknown electrical conductivity coefficients and inverse scattering problems for determining unknown acoustic velocity coefficients. In this context, the stability estimate means is estimating the continuous dependence of the determination on the measurement data of unknown coefficients. For example, in the case of EIT, the problem is to determine the unknown electrical conductivity coefficient of the EIT electrical conductivity equation using measurement data, that is, voltage (input) and current (output), on the boundary of the domain in which this equation is defined. In mathematics, infinite measurements are allowed. Alessandrini~\cite{Alessandrini} showed that logarithmic stability estimate is optimal when there are no specific assumptions about the unknown coefficients. On the other hand, it is known that Lipschitz stability estimate holds for example if the unknown coefficients are piecewise constant, such as when the unknown coefficients can be represented by a finite number of unknown parameters (Alessandrini-de Hoop-Gaburro~\cite{AHG}, Alessandrini-de Hoop-Gaburro-Sincich~\cite{AHGS17}, Alessandrini-Vessella~\cite{AV}, Beretta-de Hoop-Qiu~\cite{BHQ}, Harrach~\cite{Harrach} and the references therein). It should be noted here that even though the number of unknown parameters is finite, this Lipschitz stability estimate was obtained by using infinite measurements. Hence, there is a very natural question, how to get these results under finite measurements. Alberti-Santacesaria~\cite{AS} systematically solved this question. Namely, for a class of infinite-dimensional inverse problems that satisfy Lipschitz stability estimate, they proved that inverse problems with finite measurements also satisfy Lipschitz stability estimate. Furthermore, they introduced a global reconstruction algorithm that combines the Lipschitz stability estimate derived in \cite{AS} with the results from \cite{dHQS}. Here, let us provide an additional explanation of the meaning of the mentioned global reconstruction algorithm. We first note that the convergence proved in \cite{dHQS} is local, as the initial data needed to start the iteration for the regularization method must be selected close to the solution by some initial guess. In contrast, \cite{AS} provides a method to choose the initial data without any initial geuess. We refer the cited work as inverse problems with finite measurements abbreviated by IFM.

Now, we mention some results for the LM-method. Its theoretical study is thoroughly given in Hanke~\cite{Hanke} and \cite{KNS} assuming the tangential cone condition. As for its numerical study, especially the numerical speed of convergence, the LM-method is superior than the Landweber iteration (see Hohage~\cite{Hohage}). 

Considering the results of these previous studies, we setup our aim of study in this paper as follows. Namely, it is to study the LM-method under the mentioned H\"older stability estimate and apply it to the IFM replacing the Landweber iteration by the LM-method so that we can improve the result of Alberti-Santacesaria~\cite{AS}.

The rest of this paper is organized as follows to achieve our aim. In Section~\ref{section:LM}, we introduce the LM-method. Section~\ref{section:exact} analyzes the convergence of the LM-method and establishes the convergence rate for exact data. For noisy data, Section~\ref{section:noisy} demonstrates convergence and provides its convergence rate. Finally, Section~\ref{section:application} is devoted to applying what we have obtained so far for the LM-method to the IFM.

\section{Preliminaries} \label{section:LM}
In this section, let us introduce the LM-method. This section provides a summary of the LM-method for noisy data to get into the main part of the paper. Any modification necessary for exact data will be given after finishing the noisy data case.

To begin with, we clarify more precisely what are the set up for the operator equation mentioned in Section~\ref{section:introduction}. Let $X$, $Y$ be real Hilbert spaces, $F: X\rightarrow Y$ be an operator with the domain~$\mathcal{D}(F)$ which is open and connected. We assume that $F$ has a continuous Fr\'{e}chet derivative. Consider the following nonlinear operator equation
\begin{equation} \label{NOE}
F(x)=y
\end{equation}
for given $y\in Y$. Let $x^{\dagger}$ generate the exact data $y$, that is, $F(x^{\dagger})=y$. We also consider the case $y$ has a noise with noise level $\delta>0$. Namely, the case obtained by replacing $y$ with $y^{\delta}$ such that
\begin{equation} \label{noise}
\|y^{\delta}-y\|_{Y}\leq\delta.
\end{equation}

The LM-method for \eqref{NOE} is an iteration method given as
\begin{equation} \label{LM}
x_{k+1}^{\delta}=x_{k}^{\delta}+(F'(x_{k}^{\delta})^{\ast}F'(x_{k}^{\delta})+\alpha_{k}I)^{-1}F'(x_{k}^{\delta})^{\ast}(y^{\delta}-F(x_{k}^{\delta})),\hspace{0.1in}k\in\mathbb{N}_{0}:=\mathbb{N}\cup\{0\}
\end{equation}
with an initial data~$x_{0}^{\delta}=x_{0}$. Here, the regularization parameter $\alpha_k>0$ is determined by Morozov's discrepancy principle. That is, the parameter~$\alpha_{k}>0$ is defined to satisfy
\begin{equation} \label{MDP}
\alpha_{k}\|(F'(x_{k}^{\delta})F'(x_{k}^{\delta})^{\ast}+\alpha_{k}I)^{-1}(y^{\delta}-F(x_{k}^{\delta}))\|_{Y}=q\|y^{\delta}-F(x_{k}^{\delta})\|_{Y}
\end{equation}
for a fixed $q\in(0, 1)$. We note that there exists a solution~$\alpha_{k}$ to \eqref{MDP} uniquely provided
\begin{equation} \label{unique;alpha}
\|y^{\delta}-F(x_{k}^{\delta})-F'(x_{k}^{\delta})(x^{\dagger}-x_{k}^{\delta})\|_{Y}\leq\frac{q}{\omega}\|y^{\delta}-F(x_{k}^{\delta})\|_{Y}
\end{equation}
holds for some $\omega>1$, see \cite[p.\,65]{KNS} for instance. We also point out that we present some properties that can be deduced under the assumption \eqref{unique;alpha} in this section. However, in the subsequent sections, we will not directly assume \eqref{unique;alpha}. This is because, as we will prove later, \eqref{unique;alpha} can be derived from H\"{o}lder stability estimate. 

In the case where we know only noisy data, we have to stop the iteration procedure after appropriate number of steps in order to act as a regularization method. We use the discrepancy principle, that is, the iteration is stopped after $k_{\ast}=k_{\ast}(\delta, y^{\delta})$ steps when the following stopping criterion
\begin{equation} \label{DiscrepancyPrinciple}
\|y^{\delta}-F(x_{k_{\ast}}^{\delta})\|_{Y}\leq\tau\delta<\|y^{\delta}-F(x_{k}^{\delta})\|_{Y},\hspace{0.2in}0\leq k<k_{\ast},
\end{equation}
with an appropriately $\tau>1$ is satisfied. In the exact data case, delete the superscript~$\delta$, and ignore the stopping criterion.

To proceed further, we prepare some fundamental lemmas. First of all, we provide an alternative expression for \eqref{MDP}.
\begin{lemma} \label{lem:KNS;4.9}
Let $0<q<1$. We assume that \eqref{unique;alpha} holds so that $\alpha_{k}$ can be defined via \eqref{MDP}. We define $x_{k}^{\delta}$ as \eqref{LM}. Then, we have
\begin{equation} \label{MDP'}
\|y^{\delta}-F(x_{k}^{\delta})-F'(x_{k}^{\delta})(x_{k+1}^{\delta}-x_{k}^{\delta}))\|_{Y}=q\|y^{\delta}-F(x_{k}^{\delta})\|_{Y}.
\end{equation}
\end{lemma}
\begin{proof}
Multiplying
\[
F'(x_{k}^{\delta})(F'(x_{k}^{\delta})^{\ast}F'(x_{k}^{\delta})+\alpha_{k}I)=(F'(x_{k}^{\delta})F'(x_{k}^{\delta})^{\ast}+\alpha_{k}I)F'(x_{k}^{\delta})
\]
by $(F'(x_{k}^{\delta})F'(x_{k}^{\delta})^{\ast}+\alpha_{k}I)^{-1}$ from the left side and $(F'(x_{k}^{\delta})^{\ast}F'(x_{k}^{\delta})+\alpha_{k}I)^{-1}$ from the right side, we have
\begin{equation} \label{prep;KNS;4.9;1}
(F'(x_{k}^{\delta})F'(x_{k}^{\delta})^{\ast}+\alpha_{k}I)^{-1}F'(x_{k}^{\delta})=F'(x_{k}^{\delta})(F'(x_{k}^{\delta})^{\ast}F'(x_{k}^{\delta})+\alpha_{k}I)^{-1}.
\end{equation}
Transposing $x_{k}^{\delta}$ to the left side of \eqref{LM}, multiplying it by $(F'(x_{k}^{\delta})F'(x_{k}^{\delta})^{\ast}+\alpha_{k}I)F'(x_{k}^{\delta})$ from the left side and using \eqref{prep;KNS;4.9;1} yield
\begin{align} \label{prep;KNS;4.9;2}
&(F'(x_{k}^{\delta})F'(x_{k}^{\delta})^{\ast}+\alpha_{k}I)F'(x_{k}^{\delta})(x_{k+1}^{\delta}-x_{k}^{\delta}) \nonumber \\
&\hspace{0.1in}=(F'(x_{k}^{\delta})F'(x_{k}^{\delta})^{\ast}+\alpha_{k}I)F'(x_{k}^{\delta})(F'(x_{k}^{\delta})^{\ast}F'(x_{k}^{\delta})+\alpha_{k}I)^{-1}F'(x_{k}^{\delta})^{\ast}(y^{\delta}-F(x_{k}^{\delta}))\nonumber \\
&\hspace{0.1in}=F'(x_{k}^{\delta})F'(x_{k}^{\delta})^{\ast}(y^{\delta}-F(x_{k}^{\delta})).
\end{align}
Using \eqref{prep;KNS;4.9;2}, we have
\begin{align} \label{prep;KNS;4.9;3}
&(F'(x_{k}^{\delta})F'(x_{k}^{\delta})^{\ast}+\alpha_{k}I)(y^{\delta}-F(x_{k}^{\delta})-F'(x_{k}^{\delta})(x_{k+1}^{\delta}-x_{k}^{\delta})) \nonumber \\
&\hspace{0.1in}=(F'(x_{k}^{\delta})F'(x_{k}^{\delta})^{\ast}+\alpha_{k}I)(y^{\delta}-F(x_{k}^{\delta})) \nonumber \\
&\hspace{0.3in}-(F'(x_{k}^{\delta})F'(x_{k}^{\delta})^{\ast}+\alpha_{k}I)F'(x_{k}^{\delta})(x_{k+1}^{\delta}-x_{k}^{\delta})) \nonumber \\
&\hspace{0.1in}=(F'(x_{k}^{\delta})F'(x_{k}^{\delta})^{\ast}+\alpha_{k}I)(y^{\delta}-F(x_{k}^{\delta}))-F'(x_{k}^{\delta})F'(x_{k}^{\delta})^{\ast}(y^{\delta}-F(x_{k}^{\delta})) \nonumber \\
&\hspace{0.1in}=\alpha_{k}(y^{\delta}-F(x_{k}^{\delta})).
\end{align}
Multiplying \eqref{prep;KNS;4.9;3} by $(F'(x_{k}^{\delta})F'(x_{k}^{\delta})^{\ast}+\alpha_{k}I)^{-1}$ from the left side yields
\begin{equation} \label{KNS;4.9}
y^{\delta}-F(x_{k}^{\delta})-F'(x_{k}^{\delta})(x_{k+1}^{\delta}-x_{k}^{\delta})=\alpha_{k}(F'(x_{k}^{\delta})F'(x_{k}^{\delta})^{\ast}+\alpha_{k}I)^{-1}(y^{\delta}-F(x_{k}^{\delta})).
\end{equation}
Then, \eqref{MDP'} follows from \eqref{MDP} and \eqref{KNS;4.9}.
\end{proof}

\begin{lemma}[\textup{monotonicity of error, \cite[Proposition 4.1]{KNS}}] \label{prop:KNS;4.5}
Let $0<q<1$. We assume that \eqref{NOE} has a solution~$x^{\dagger}$ and that \eqref{unique;alpha} holds so that $\alpha_{k}$ can be defined via \eqref{MDP}. We define $x_{k}^{\delta}$ as \eqref{LM}. Then, the following estimate holds:
\begin{equation} \label{KNS;4.5}
\|x_{k}^{\delta}-x^{\dagger}\|_{X}^{2}-\|x_{k+1}^{\delta}-x^{\dagger}\|_{X}^{2}>\|x_{k+1}^{\delta}-x_{k}^{\delta}\|_{X}^{2}.
\end{equation}
\end{lemma}
\begin{proof}
Since we have
\begin{equation} \label{prep;KNS;4.5;1}
\|x_{k+1}^{\delta}-x^{\dagger}\|_{X}^{2}-\|x_{k}^{\delta}-x^{\dagger}\|_{X}^{2}=-\|x_{k+1}^{\delta}-x_{k}^{\delta}\|_{X}^{2}+2(x_{k+1}^{\delta}-x_{k}^{\delta}, x_{k+1}^{\delta}-x^{\dagger})_{X},
\end{equation}
it suffices to examine the second term on the right-hand side of \eqref{prep;KNS;4.5;1}.
We here note that it follows from analysis similar to the derivation of \eqref{prep;KNS;4.9;1} that
\begin{equation} \label{prep;ast}
(F'(x_{k}^{\delta})^{\ast}F'(x_{k}^{\delta})+\alpha_{k}I)^{-1}F'(x_{k}^{\delta})^{\ast}=F'(x_{k}^{\delta})^{\ast}(F'(x_{k}^{\delta})F'(x_{k}^{\delta})^{\ast}+\alpha_{k}I)^{-1}.
\end{equation}
It follows from \eqref{LM}, \eqref{prep;ast} and \eqref{KNS;4.9} that
\begin{align} \label{prep;KNS;4.5;1;2nd-term-1}
&(x_{k+1}^{\delta}-x_{k}^{\delta}, x_{k+1}^{\delta}-x^{\dagger})_{X} \nonumber \\
&\hspace{0.1in}=((F'(x_{k}^{\delta})^{\ast}F'(x_{k}^{\delta})+\alpha_{k}I)^{-1}F'(x_{k}^{\delta})^{\ast}(y^{\delta}-F(x_{k}^{\delta})), x_{k+1}^{\delta}-x^{\dagger})_{X} \nonumber \\
&\hspace{0.1in}=(F'(x_{k}^{\delta})^{\ast}(F'(x_{k}^{\delta})F'(x_{k}^{\delta})^{\ast}+\alpha_{k}I)^{-1}(y^{\delta}-F(x_{k}^{\delta})), x_{k+1}^{\delta}-x^{\dagger})_{X} \nonumber \\
&\hspace{0.1in}=((F'(x_{k}^{\delta})F'(x_{k}^{\delta})^{\ast}+\alpha_{k}I)^{-1}(y^{\delta}-F(x_{k}^{\delta})), F'(x_{k}^{\delta})(x_{k+1}^{\delta}-x^{\dagger}))_{Y} \nonumber \\
&\hspace{0.1in}=\frac{1}{\alpha_{k}}(y^{\delta}-F(x_{k}^{\delta})-F'(x_{k}^{\delta})(x_{k+1}^{\delta}-x_{k}^{\delta}), F'(x_{k}^{\delta})(x_{k+1}^{\delta}-x^{\dagger}))_{Y}.
\end{align}
Moreover, on the inner product of \eqref{prep;KNS;4.5;1;2nd-term-1}, we have
\begin{align} \label{prep;KNS;4.5;1;2nd-term-2}
&(y^{\delta}-F(x_{k}^{\delta})-F'(x_{k}^{\delta})(x_{k+1}^{\delta}-x_{k}^{\delta}), F'(x_{k}^{\delta})(x_{k+1}^{\delta}-x^{\dagger}))_{Y} \nonumber \\
&\hspace{0.1in}=(y^{\delta}-F(x_{k}^{\delta})-F'(x_{k}^{\delta})(x_{k+1}^{\delta}-x_{k}^{\delta}), \nonumber \\
&\hspace{0.3in}-(y^{\delta}-F(x_{k}^{\delta})-F'(x_{k}^{\delta})(x_{k+1}^{\delta}-x_{k}^{\delta}))+y^{\delta}-F(x_{k}^{\delta})-F'(x_{k}^{\delta})(x^{\dagger}-x_{k}^{\delta}))_{Y} \nonumber \\
&\hspace{0.1in}=-\|y^{\delta}-F(x_{k}^{\delta})-F'(x_{k}^{\delta})(x_{k+1}^{\delta}-x_{k}^{\delta})\|_{Y}^{2} \nonumber \\
&\hspace{0.3in}+(y^{\delta}-F(x_{k}^{\delta})-F'(x_{k}^{\delta})(x_{k+1}^{\delta}-x_{k}^{\delta}), y^{\delta}-F(x_{k}^{\delta})-F'(x_{k}^{\delta})(x^{\dagger}-x_{k}^{\delta}))_{Y} \nonumber \\
&\hspace{0.1in}\leq-\|y^{\delta}-F(x_{k}^{\delta})-F'(x_{k}^{\delta})(x_{k+1}^{\delta}-x_{k}^{\delta})\|_{Y}^{2} \nonumber \\
&\hspace{0.3in}+\|y^{\delta}-F(x_{k}^{\delta})-F'(x_{k}^{\delta})(x_{k+1}^{\delta}-x_{k}^{\delta})\|_{Y}\|y^{\delta}-F(x_{k}^{\delta})-F'(x_{k}^{\delta})(x^{\dagger}-x_{k}^{\delta})\|_{Y} \nonumber \\
&\hspace{0.1in}=-\|y^{\delta}-F(x_{k}^{\delta})-F'(x_{k}^{\delta})(x_{k+1}^{\delta}-x_{k}^{\delta})\|_{Y} \nonumber \\
&\hspace{0.5in}\times(\|y^{\delta}-F(x_{k}^{\delta})-F'(x_{k}^{\delta})(x_{k+1}^{\delta}-x_{k}^{\delta})\|_{Y} \nonumber \\
&\hspace{1.2in}-\|y^{\delta}-F(x_{k}^{\delta})-F'(x_{k}^{\delta})(x^{\dagger}-x_{k}^{\delta})\|_{Y}).
\end{align}
It follows from \eqref{unique;alpha}, \eqref{MDP'} and $\omega>1$ that
\begin{align} \label{prep;KNS;4.5;1;2nd-term-3}
&\|y^{\delta}-F(x_{k}^{\delta})-F'(x_{k}^{\delta})(x^{\dagger}-x_{k}^{\delta})\|_{Y} \nonumber \\
&\hspace{0.1in}\leq\frac{q}{\omega}\|y^{\delta}-F(x_{k}^{\delta})\|_{Y}=\frac{1}{\omega}\|y^{\delta}-F(x_{k}^{\delta})-F'(x_{k}^{\delta})(x_{k+1}^{\delta}-x_{k}^{\delta})\|_{Y} \nonumber \\
&\hspace{0.1in}<\|y^{\delta}-F(x_{k}^{\delta})-F'(x_{k}^{\delta})(x_{k+1}^{\delta}-x_{k}^{\delta})\|_{Y}.
\end{align}
Combining \eqref{prep;KNS;4.5;1;2nd-term-1}, \eqref{prep;KNS;4.5;1;2nd-term-2} and \eqref{prep;KNS;4.5;1;2nd-term-3} implies $(x_{k+1}^{\delta}-x_{k}^{\delta}, x_{k+1}^{\delta}-x^{\dagger})_{X}<0$, which in turn yields \eqref{KNS;4.5}.
\end{proof}

Finally, we introduce the well-known inequality, see the proof of \cite[Proposition 4.1]{KNS}, for instance.
\begin{lemma} \label{lem:KNS;Prop;4.1}
Let $0<q<1$. We assume that \eqref{unique;alpha} holds so that $\alpha_{k}$ can be defined via \eqref{MDP}. We define $x_{k}^{\delta}$ as \eqref{LM}. Then, we have
\begin{equation} \label{alpha_k}
\alpha_{k}\leq\frac{q}{1-q}\|F'(x_{k}^{\delta})\|_{\mathcal{L}(X, Y)}^{2}.
\end{equation}
\end{lemma}
\begin{proof}
We first remark that we have
\begin{align} \label{prep;alpha_k;1}
&\|y^{\delta}-F(x_{k}^{\delta})\|_{Y} \nonumber \\
&\hspace{0.1in}=\|(F'(x_{k}^{\delta})F'(x_{k}^{\delta})^{\ast}+\alpha_{k}I)(F'(x_{k}^{\delta})F'(x_{k}^{\delta})^{\ast}+\alpha_{k}I)^{-1}(y^{\delta}-F(x_{k}^{\delta}))\|_{Y} \nonumber \\
&\hspace{0.1in}\leq(\|F'(x_{k}^{\delta})\|_{\mathcal{L}(X, Y)}^{2}+\alpha_{k})\|(F'(x_{k}^{\delta})F'(x_{k}^{\delta})^{\ast}+\alpha_{k}I)^{-1}(y^{\delta}-F(x_{k}^{\delta}))\|_{Y}.
\end{align}
Using \eqref{MDP} and \eqref{prep;alpha_k;1}, we obtain
\begin{align} \label{prep;alpha_k;2}
q\|y^{\delta}-F(x_{k}^{\delta})\|_{Y}&=\alpha_{k}\|(F'(x_{k}^{\delta})F'(x_{k}^{\delta})^{\ast}+\alpha_{k}I)^{-1}(y^{\delta}-F(x_{k}^{\delta}))\|_{Y} \nonumber \\
&\geq\frac{\alpha_{k}}{\|F'(x_{k}^{\delta})\|_{\mathcal{L}(X, Y)}^{2}+\alpha_{k}}\|y^{\delta}-F(x_{k}^{\delta})\|_{Y}.
\end{align}
Then, estimate~\eqref{alpha_k} immediately follows from \eqref{prep;alpha_k;2}.
\end{proof}

\section{Convergence for exact data} \label{section:exact}
In this section, we study the convergence for exact data. For the operator~$F$, we assume the following.
\begin{assumption} \label{asm;Holder}
Let $B=\{x\in X \mid \frac{1}{2}\|x-x_{0}\|_{X}^{2}\leq\rho'\}\subset\mathcal{D}(F)$.
\begin{enumerate}
\renewcommand{\labelenumi}{(\alph{enumi})}
\setlength{\leftskip}{-0.1in}
\setlength{\itemsep}{-4pt}
\item $F$ is Fr\'{e}chet derivative in $B$, and its derivative $F'(\cdot)$ is Lipschitz continuous with a Lipschitz constant $L>0$ such that
\begin{equation} \label{F'-Lip}
\|F'(x)-F'(\widetilde{x})\|_{\mathcal{L}(X, Y)}\leq L\|x-\widetilde{x}\|_{X},\hspace{0.2in}x, \widetilde{x}\in B.
\end{equation}
\item There exists $\widehat{L}>0$ such that
\begin{equation} \label{F'-bounded}
\|F'(x)\|_{\mathcal{L}(X, Y)}\leq\widehat{L},\hspace{0.2in}x\in B.
\end{equation}
\item Equation~\eqref{NOE} has the uniform H\"{o}lder stability estimate given as
\begin{equation} \label{F;inv-Holder}
\frac{1}{\sqrt{2}}\|x-\widetilde{x}\|_{X}\leq C_{F}\|F(x)-F(\widetilde{x})\|_{Y}^{\frac{1+\varepsilon}{2}},\hspace{0.2in}x, \widetilde{x}\in B
\end{equation}
for some $C_{F}>0$ and $\varepsilon\in(0, 1]$.
\end{enumerate}
\end{assumption}
\begin{remark} \label{rmk:TCC;Holder}
It follows from Assumption~\ref{asm;Holder} that the tangential cone condition~\cite{KSS}
\begin{equation} \label{TCC}
\|F(x)-F(\widetilde{x})-F'(x)(x-\widetilde{x})\|_{Y}\leq\eta\|F(x)-F(\widetilde{x})\|_{Y}\hspace{0.2in}x, \widetilde{x}\in B
\end{equation}
for some $0<\eta<1$ if $\rho'>0$ is sufficiently small. Indeed, applying the fundamental theorem of calculus for the Fr\'{e}chet derivative~\cite[Lemma\,A.63]{Kirsch}, \eqref{F'-Lip} and \eqref{F;inv-Holder} yields
\begin{align*} 
\|F(x)-F(\widetilde{x})-F'(\widetilde{x})(x-\widetilde{x})\|_{Y}&\leq\frac{L}{2}\|x-\widetilde{x}\|_{X}^{2} \\
&\leq\frac{L}{2}\|x-\widetilde{x}\|_{X}^{\frac{2\varepsilon}{1+\varepsilon}}(\sqrt{2}C_{F})^{\frac{2}{1+\varepsilon}}\|F(x)-F(\widetilde{x})\|_{Y} \\
&\leq\frac{L}{2}(2\sqrt{2\rho'})^{\frac{2\varepsilon}{1+\varepsilon}}(\sqrt{2}C_{F})^{\frac{2}{1+\varepsilon}}\|F(x)-F(\widetilde{x})\|_{Y}.
\end{align*}
\end{remark}
We are now ready to state our main result for exact data explicitly. It says that the sequence~$\{x_{k}\}_{k=0}^{\infty}$ obtained by the LM-method for the exact data~$y$ converges to a solution~$x^{\dagger}$ with a convergence rate.
\begin{theorem} \label{thm:exact;Hilbert}
We assume that there exists a solution $x^{\dagger}$ to the equation \eqref{NOE} and that Assumption~\ref{asm;Holder} holds. Let $q$ satisfy
\begin{equation} \label{q}
q(1-q)<2\widehat{L}^{2}C_{F}^{\frac{4}{1+\varepsilon}}
\end{equation}
and let $\rho$ satisfy
\begin{equation} \label{rho}
\rho=\frac{1}{2\widehat{L}^{2}}\left(\frac{q}{2LC_{F}^{2}}\right)^{\frac{2}{\varepsilon}}<\rho'. 
\end{equation}
If
\begin{equation} \label{x0}
\frac{1}{2}\|x_{0}-x^{\dagger}\|_{X}^{2}\leq\rho,
\end{equation}
then the iterates satisfy
\begin{equation} \label{xk}
\frac{1}{2}\|x_{k}-x^{\dagger}\|_{X}^{2}\leq\rho,\hspace{0.3in}k\in\mathbb{N}
\end{equation}
and $x_{k}\to x^{\dagger}$ as $k\to\infty$. Moreover, let
\begin{equation} \label{c}
c=\frac{q(1-q)}{2\widehat{L}^{2}C_{F}^{\frac{4}{1+\varepsilon}}}.
\end{equation}
Then, it follows from \eqref{q} that $0<c<1$. The convergence rate is given by 
\begin{equation} \label{conv-rate;L}
\frac{1}{2}\|x_{k}-x^{\dagger}\|_{X}^{2}\leq\rho(1-c)^{k}
\end{equation}
if $\varepsilon=1$. For $\varepsilon\in(0, 1)$, the convergence rate is given by
\begin{equation} \label{conv-rate;H}
\frac{1}{2}\|x_{k}-x^{\dagger}\|_{X}^{2}\leq\left(ck\frac{1-\varepsilon}{1+\varepsilon}+\rho^{-\frac{1-\varepsilon}{1+\varepsilon}}\right)^{-\frac{1+\varepsilon}{1-\varepsilon}}.
\end{equation}
\end{theorem}
\begin{proof}
We first see that the parameter~$\alpha_{0}$ is determined uniquely. It suffices to confirm that \eqref{unique;alpha} is satisfied. Applying \cite[Lemma\,A.63]{Kirsch} and using \eqref{F;inv-Holder} yield
\begin{equation} \label{prep;unique;alpha;exact;1}
\|y-F(x_{0})-F'(x_{0})(x^{\dagger}-x_{0})\|_{Y}\leq\frac{L}{2}\|x_{0}-x^{\dagger}\|_{X}^{2}\leq LC_{F}^{2}\|F(x_{0})-y\|_{Y}^{1+\varepsilon}.
\end{equation}
By the mean value inequality, \eqref{F'-bounded}, \eqref{x0} and \eqref{rho}, we have
\begin{equation} \label{prep;unique;alpha;exact;2}
\|F(x_{0})-y\|_{Y}\leq\widehat{L}\|x_{0}-x^{\dagger}\|_{X}\leq\sqrt{2\rho}\widehat{L}=\left(\frac{q}{2LC_{F}^{2}}\right)^{\frac{1}{\varepsilon}}. 
\end{equation}
It follows from \eqref{prep;unique;alpha;exact;1} and \eqref{prep;unique;alpha;exact;2} that
\begin{equation} \label{unique;alpha;0}
\|y-F(x_{0})-F'(x_{0})(x^{\dagger}-x_{0})\|_{Y}\leq\frac{q}{2}\|F(x_{0})-y\|_{Y}.
\end{equation}
Hence, \eqref{unique;alpha} holds with $\omega=2$.

Now, assuming that $\alpha_{k}$ is unique, we proceed and will prove this property later. Using \eqref{LM}, we have
\begin{align} \label{gamma_k+1-gamma_k}
&\frac{1}{2}\left(\|x_{k+1}-x^{\dagger}\|_{X}^{2}-\|x_{k}-x^{\dagger}\|_{X}^{2}\right) \nonumber \\
&\hspace{0.1in}=\frac{1}{2}\|x_{k+1}-x_{k}\|_{X}^{2}+(x_{k+1}-x_{k}, x_{k}-x^{\dagger})_{X} \nonumber \\
&\hspace{0.1in}=\frac{1}{2}\|(F'(x_{k})^{\ast}F'(x_{k})+\alpha_{k}I)^{-1}F'(x_{k})^{\ast}(y-F(x_{k}))\|_{X}^{2} \nonumber \\
&\hspace{0.3in}+((F'(x_{k})^{\ast}F'(x_{k})+\alpha_{k}I)^{-1}F'(x_{k})^{\ast}(y-F(x_{k})), x_{k}-x^{\dagger})_{X}.
\end{align}
We see each term on the right-hand side of \eqref{gamma_k+1-gamma_k}. It follows from \eqref{prep;ast} that 
\begin{align} \label{1st-term}
&\|(F'(x_{k})^{\ast}F'(x_{k})+\alpha_{k}I)^{-1}F'(x_{k})^{\ast}(y-F(x_{k}))\|_{X}^{2} \nonumber \\
&\hspace{0.1in}=\|F'(x_{k})^{\ast}(F'(x_{k})F'(x_{k})^{\ast}+\alpha_{k}I)^{-1}(y-F(x_{k}))\|_{X}^{2}.
\end{align}
We next estimate the second term on the right-hand side of \eqref{gamma_k+1-gamma_k}. We first note that we have
\begin{align} \label{prep;2nd-term}
&((F'(x_{k})^{\ast}F'(x_{k})+\alpha_{k}I)^{-1}F'(x_{k})^{\ast}(y-F(x_{k})), x_{k}-x^{\dagger})_{X} \nonumber \\
&\hspace{0.1in}=(F'(x_{k})^{\ast}(F'(x_{k})F'(x_{k})^{\ast}+\alpha_{k}I)^{-1}(y-F(x_{k})), x_{k}-x^{\dagger})_{X} \nonumber \\
&\hspace{0.1in}=((F'(x_{k})F'(x_{k})^{\ast}+\alpha_{k}I)^{-1}(y-F(x_{k})), -F'(x_{k})(x^{\dagger}-x_{k}))_{Y} \nonumber \\
&\hspace{0.1in}=-((F'(x_{k})F'(x_{k})^{\ast}+\alpha_{k}I)^{-1}(y-F(x_{k})), y-F(x_{k}))_{Y} \nonumber \\
&\hspace{0.3in}+((F'(x_{k})F'(x_{k})^{\ast}+\alpha_{k}I)^{-1}(y-F(x_{k})), y-F(x_{k})-F'(x_{k})(x^{\dagger}-x_{k}))_{Y}
\end{align}
by \eqref{prep;ast}. For the first inner product on the right-hand side of \eqref{prep;2nd-term},
\begin{align} \label{prep;2nd-term;1}
&((F'(x_{k})F'(x_{k})^{\ast}+\alpha_{k}I)^{-1}(y-F(x_{k})), y-F(x_{k}))_{Y} \nonumber \\
&\hspace{0.1in}=((F'(x_{k})F'(x_{k})^{\ast}+\alpha_{k}I)^{-1}(y-F(x_{k})), \nonumber \\
&\hspace{0.5in}(F'(x_{k})F'(x_{k})^{\ast}+\alpha_{k}I)(F'(x_{k})F'(x_{k})^{\ast}+\alpha_{k}I)^{-1}(y-F(x_{k})))_{Y} \nonumber \\
&\hspace{0.1in}=((F'(x_{k})F'(x_{k})^{\ast}+\alpha_{k}I)^{-1}(y-F(x_{k})), \nonumber \\
&\hspace{0.5in}F'(x_{k})F'(x_{k})^{\ast}(F'(x_{k})F'(x_{k})^{\ast}+\alpha_{k}I)^{-1}(y-F(x_{k})))_{Y} \nonumber \\
&\hspace{0.3in}+\alpha_{k}((F'(x_{k})F'(x_{k})^{\ast}+\alpha_{k}I)^{-1}(y-F(x_{k})), \nonumber \\
&\hspace{1in}(F'(x_{k})F'(x_{k})^{\ast}+\alpha_{k}I)^{-1}(y-F(x_{k})))_{Y} \nonumber \\
&\hspace{0.1in}=\|F'(x_{k})^{\ast}(F'(x_{k})F'(x_{k})^{\ast}+\alpha_{k}I)^{-1}(y-F(x_{k}))\|_{X}^{2} \nonumber \\
&\hspace{0.3in}+\alpha_{k}\|(F'(x_{k})F'(x_{k})^{\ast}+\alpha_{k}I)^{-1}(y-F(x_{k}))\|_{Y}^{2}
\end{align}
holds. Moreover, it follows from \eqref{MDP} that 
\begin{equation} \label{prep;alpha;nu}
\|(F'(x_{k})F'(x_{k})^{\ast}+\alpha_{k}I)^{-1}(y-F(x_{k}))\|_{Y}=\frac{q}{\alpha_{k}}\|y-F(x_{k})\|_{Y}.
\end{equation}
Thus, on the first inner product on the right-hand side of \eqref{prep;2nd-term}, we have
\begin{align} \label{prep;2nd-term;1-2}
&((F'(x_{k})F'(x_{k})^{\ast}+\alpha_{k}I)^{-1}(y-F(x_{k})), y-F(x_{k}))_{Y} \nonumber \\
&\hspace{0.1in}=\|F'(x_{k})^{\ast}(F'(x_{k})F'(x_{k})^{\ast}+\alpha_{k}I)^{-1}(y-F(x_{k}))\|_{X}^{2}+\frac{q^{2}}{\alpha_{k}}\|y-F(x_{k})\|_{Y}^{2}
\end{align}
by \eqref{prep;2nd-term;1} and \eqref{prep;alpha;nu}. Concerning the second inner product on the right-hand side of \eqref{prep;2nd-term}, using the Cauchy-Schwarz inequality and \eqref{prep;alpha;nu} yields
\begin{align} \label{prep;2nd-term;2}
&((F'(x_{k})F'(x_{k})^{\ast}+\alpha_{k}I)^{-1}(y-F(x_{k})), y-F(x_{k})-F'(x_{k})(x^{\dagger}-x_{k}))_{Y}  \nonumber \\
&\hspace{0.1in}\leq\|(F'(x_{k})F'(x_{k})^{\ast}+\alpha_{k}I)^{-1}(y-F(x_{k}))\|_{Y}\|y-F(x_{k})-F'(x_{k})(x^{\dagger}-x_{k}))\|_{Y} \nonumber \\
&\hspace{0.1in}=\frac{q}{\alpha_{k}}\|y-F(x_{k})\|_{Y}\|y-F(x_{k})-F'(x_{k})(x^{\dagger}-x_{k})\|_{Y}.
\end{align}
Combining \eqref{prep;2nd-term}, \eqref{prep;2nd-term;1-2} and \eqref{prep;2nd-term;2}, we have
\begin{align} \label{2nd-term}
&((F'(x_{k})^{\ast}F'(x_{k})+\alpha_{k}I)^{-1}F'(x_{k})^{\ast}(y-F(x_{k})), x_{k}-x^{\dagger})_{X} \nonumber \\
&\hspace{0.1in}\leq-\|F'(x_{k})^{\ast}(F'(x_{k})F'(x_{k})^{\ast}+\alpha_{k}I)^{-1}(y-F(x_{k}))\|_{X}^{2}-\frac{q^{2}}{\alpha_{k}}\|y-F(x_{k})\|_{Y}^{2} \nonumber \\
&\hspace{0.3in}+\frac{q}{\alpha_{k}}\|y-F(x_{k})\|_{Y}\|y-F(x_{k})-F'(x_{k})(x^{\dagger}-x_{k})\|_{Y}.
\end{align}
Substituting \eqref{1st-term} and \eqref{2nd-term} into \eqref{gamma_k+1-gamma_k} and using the notation 
\[
\gamma_{k}=\frac{1}{2}\|x_{k}-x^{\dagger}\|_{X}^{2},
\]
we have
\begin{align} \label{k+1}
\gamma_{k+1}-\gamma_{k}&\leq-\frac{1}{2}\|F'(x_{k})^{\ast}(F'(x_{k})F'(x_{k})^{\ast}+\alpha_{k}I)^{-1}(y-F(x_{k}))\|_{X}^{2} \nonumber \\
&\hspace{0.3in}-\frac{q^{2}}{\alpha_{k}}\|y-F(x_{k})\|_{Y}^{2} \nonumber \\
&\hspace{0.3in}+\frac{q}{\alpha_{k}}\|y-F(x_{k})\|_{Y}\|y-F(x_{k})-F'(x_{k})(x^{\dagger}-x_{k})\|_{Y}.
\end{align}
We now prove
\begin{equation} \label{induction}
\gamma_{k}\leq\rho
\end{equation}
and
\begin{equation} \label{unique;alpha;k}
\|y-F(x_{k})-F'(x_{k})(x^{\dagger}-x_{k})\|_{Y}\leq\frac{q}{2}\|F(x_{k})-y\|_{Y}
\end{equation}
by induction. We remark that the estimate \eqref{unique;alpha;k} is \eqref{unique;alpha} with $\omega=2$. Hence, \eqref{unique;alpha;k} implies that $\alpha_{k}$ satisfying \eqref{MDP} is unique. We can also define $x_{k+1}$ by \eqref{LM}. First, it follows from \eqref{xk} and \eqref{unique;alpha;0} that \eqref{induction} and \eqref{unique;alpha;k} hold for $k=0$. Now, we assume that \eqref{induction} and \eqref{unique;alpha;k} for $k=m$. In particular, we can define $\alpha_{m}$ and $x_{m+1}$ by \eqref{MDP} and \eqref{LM}, respectively. We here remark that $x_{m}, x^{\dagger}\in B$ by \eqref{x0} and the assumption of induction. Therefore, applying \eqref{F'-Lip}, \cite[Lemma\,A.63]{Kirsch} and \eqref{F;inv-Holder} yields
\begin{equation} \label{prep;FTftFD}
\|y-F(x_{m})-F'(x_{m})(x^{\dagger}-x_{m})\|_{Y}\leq\frac{L}{2}\|x_{m}-x^{\dagger}\|_{X}^{2}\leq LC_{F}^{2}\|F(x_{m})-y\|_{Y}^{1+\varepsilon}.
\end{equation}
By the mean value inequality, \eqref{F'-bounded}, the assumption of induction and \eqref{rho}, we have
\begin{equation} \label{MVI}
\|F(x_{m})-y\|_{Y}\leq\widehat{L}\|x_{m}-x^{\dagger}\|_{X}\leq\sqrt{2\rho}\widehat{L}=\left(\frac{q}{2LC_{F}^{2}}\right)^{\frac{1}{\varepsilon}}.
\end{equation}
Then, substituting \eqref{prep;FTftFD} into \eqref{k+1} and using \eqref{MVI} imply
\begin{align} \label{m+1;prep;2}
\gamma_{m+1}-\gamma_{m}&\leq-\frac{1}{2}\|F'(x_{m})^{\ast}(F'(x_{m})F'(x_{m})^{\ast}+\alpha_{m}I)^{-1}(y-F(x_{m}))\|_{X}^{2} \nonumber \\
&\hspace{0.3in}-\frac{q^{2}}{\alpha_{m}}\|y-F(x_{m})\|_{Y}^{2}+\frac{q LC_{F}^{2}}{\alpha_{m}}\|y-F(x_{m})\|_{Y}^{2+\varepsilon} \nonumber \\
&\leq-\frac{q^{2}}{\alpha_{m}}\left(1-\frac{LC_{F}^{2}}{q}\|y-F(x_{m})\|_{Y}^{\varepsilon}\right)\|y-F(x_{m})\|_{Y}^{2} \nonumber \\
&\leq-\frac{q^{2}}{2\alpha_{m}}\|y-F(x_{m})\|_{Y}^{2}.
\end{align}
It follows from \eqref{m+1;prep;2} that $\gamma_{m+1}-\gamma_{m}\leq0$. Hence, we have 
\begin{equation} \label{MD}
\gamma_{m+1}\leq\gamma_{m}\leq\rho.
\end{equation}
Next, applying \eqref{F'-Lip}, \cite[Lemma\,A.63]{Kirsch} and \eqref{F;inv-Holder} yields
\begin{equation} \label{prep;FTftFD;m+1}
\|y-F(x_{m+1})-F'(x_{m+1})(x^{\dagger}-x_{m+1})\|_{Y}\leq\frac{L}{2}\|x_{m+1}-x^{\dagger}\|_{X}^{2}\leq LC_{F}^{2}\|F(x_{m+1})-y\|_{Y}^{1+\varepsilon}.
\end{equation}
in the same way as obtaining \eqref{prep;FTftFD}. Similarly to \eqref{MVI}, using the mean value inequality, \eqref{F'-bounded}, \eqref{MD} and \eqref{rho}, we have 
\begin{equation} \label{MVI;m+1}
\|F(x_{m+1})-y\|_{Y}\leq\widehat{L}\|x_{m+1}-x^{\dagger}\|_{X}\leq\sqrt{2\rho}\widehat{L}=\left(\frac{q}{2LC_{F}^{2}}\right)^{\frac{1}{\varepsilon}}.
\end{equation}
Combining \eqref{prep;FTftFD;m+1} and \eqref{MVI;m+1}, \eqref{unique;alpha;k} holds for $k=m+1$.

Therefore, by induction, it can be shown that $\alpha_{k}$, and consequently $x_{k+1}$, is well-defined, and furthermore \eqref{induction} holds for all non-negative integers $k$.

Since $\{\gamma_{k}\}_{k=0}^{\infty}$ is monotonic decreasing and bounded below by \eqref{MD}, it is convergent. We note that we have
\begin{equation} \label{alpha_k;L}
\alpha_{k}\leq\frac{q\widehat{L}^{2}}{1-q}
\end{equation}
from \eqref{alpha_k} and \eqref{F'-bounded}. Moreover, using \eqref{m+1;prep;2}, \eqref{alpha_k;L}, \eqref{F;inv-Holder} and \eqref{c}, we obtain
\begin{align} \label{order;k+1}
\gamma_{k+1}-\gamma_{k}&\leq-\frac{q^{2}}{2\alpha_{k}}\|y-F(x_{k})\|_{Y}^{2}\leq-\frac{q(1-q)}{2\widehat{L}^{2}}\|y-F(x_{k})\|_{Y}^{2} \nonumber \\
&\leq-\frac{q(1-q)}{2\widehat{L}^{2}C_{F}^{\frac{4}{1+\varepsilon}}}\gamma_{k}^{\frac{2}{1+\varepsilon}}=-c\gamma_{k}^{\frac{2}{1+\varepsilon}}.
\end{align}
It follows from \eqref{c} that $0<c<1$. By letting $k$ go to infinity on both sides of \eqref{order;k+1}, we conclude that $\gamma_{k}\to 0$ as $k\to\infty$. Furthermore, in the case where $\varepsilon=1$, the estimate \eqref{order;k+1} yields the convergence rate \eqref{conv-rate;L}.

For the convergence rate with $0<\varepsilon<1$, since
\[
\gamma_{k+1}\leq\gamma_{k}-c\gamma_{k}^{\frac{2}{1+\varepsilon}}=\gamma_{k}(1-c\gamma_{k}^{\frac{1-\varepsilon}{1+\varepsilon}})
\]
holds by \eqref{order;k+1}, we obtain
\begin{equation} \label{prep;order}
\gamma_{k+1}^{-\frac{1-\varepsilon}{1+\varepsilon}}\geq\gamma_{k}^{-\frac{1-\varepsilon}{1+\varepsilon}}(1-c\gamma_{k}^{\frac{1-\varepsilon}{1+\varepsilon}})^{-\frac{1-\varepsilon}{1+\varepsilon}}\geq\gamma_{k}^{-\frac{1-\varepsilon}{1+\varepsilon}}+c\frac{1-\varepsilon}{1+\varepsilon},
\end{equation}
where we use $(1-s)^{-\beta}\geq1+\beta s$ for $0\leq s<1$ and $0\leq\beta\leq 1$. Hence, from \eqref{prep;order} and \eqref{x0}, we have
\[
\gamma_{k}\leq\left(\gamma_{0}^{-\frac{1-\varepsilon}{1+\varepsilon}}+c\frac{1-\varepsilon}{1+\varepsilon}k\right)^{{-\frac{1+\varepsilon}{1-\varepsilon}}}\leq\left(\rho^{-\frac{1-\varepsilon}{1+\varepsilon}}+c\frac{1-\varepsilon}{1+\varepsilon}k\right)^{{-\frac{1+\varepsilon}{1-\varepsilon}}}.
\]
The proof is complete.
\end{proof}

\begin{remark} \label{rmk:mono;exact}
Since it follows from \eqref{KNS;4.5} that $\|x_{k}-x^{\dagger}\|_{X}^{2}$ decreases monotonically, one might think that \eqref{xk} could be easily derived by combining it with \eqref{x0}. However, since we would like to prove convergence and obtain the convergence rate in Theorem~\ref{thm:exact;Hilbert}, a different approach is required.
\end{remark}

\section{Analysis for noisy data} \label{section:noisy}
In this section, we study the noisy data case under the same conditions on the operator~$F$ as the exact data case. We construct the sequence~$\{x_{k}^{\delta}\}$ by the LM-method for noisy data~$y^{\delta}$. Then, the following theorem states that we can obtain an approximate solution when we stop the iteration under the criterion~\eqref{DiscrepancyPrinciple}.
\begin{theorem} \label{thm:noisy;Hilbert}
We assume that there exists a solution $x^{\dagger}$ to the equation~\eqref{NOE} and that Assumption~\ref{asm;Holder} holds. Let
\begin{equation} \label{rho;noisy}
\rho=\frac{1}{2\widehat{L}^{2}}\left(\frac{q}{4LC_{F}^{2}}\right)^{\frac{2}{\varepsilon}}<\rho'.
\end{equation}
We assume that \eqref{DiscrepancyPrinciple} holds with $\tau$ satisfying
\begin{equation} \label{R;noisy}
R:=\frac{3}{4}-\left(\frac{1}{q}+\frac{1}{4}\right)\frac{1}{\tau}>0.
\end{equation}
If \eqref{noise} and \eqref{x0} hold, then we have
\begin{equation} \label{mono;noisy}
\frac{1}{2}\|x_{k+1}^{\delta}-x^{\dagger}\|_{X}^{2}\leq\frac{1}{2}\|x_{k}^{\delta}-x^{\dagger}\|_{X}^{2}\hspace{0.2in}\mbox{for }0\leq k<k_{\ast}.
\end{equation}
The stopping index $k_{\ast}$ is finite. Moreover, for a given $\delta>0$, if $\rho>0$ is such that 
\begin{equation} \label{asm;noisy}
\rho\leq C\delta^{2}
\end{equation}
for some $C>0$, then the following convergence rates can be derived: 
\begin{equation} \label{stab;noisy}
\frac{1}{2}\|x_{k_{\ast}}^{\delta}-x^{\dagger}\|_{X}^{2}\leq C'\delta^{2},
\end{equation}
where
\[
C'=C-\frac{q(1-q)\tau^{2}Rk_{\ast}}{\widehat{L}^{2}}.
\]
\end{theorem}
\begin{proof}
If $k_{\ast}=0$, nothing has to be shown. Hence, let us now assume that $k_{\ast}\geq1$. We first demonstrate that assumptions of Theorem~\ref{thm:noisy;Hilbert} imply \eqref{unique;alpha} for $k=0$. It follows from \eqref{noise}, \cite[Lemma\,A.63]{Kirsch} and \eqref{F;inv-Holder} that
\begin{align} \label{prep;unique;alpha;noisy;1}
&\|y^{\delta}-F(x_{0})-F'(x_{0})(x^{\dagger}-x_{0})\|_{Y} \nonumber \\
&\hspace{0.1in}\leq\delta+\|y-F(x_{0})-F'(x_{0})(x^{\dagger}-x_{0})\|_{Y} \nonumber \\
&\hspace{0.1in}\leq\delta+\frac{L}{2}\|x_{0}-x^{\dagger}\|_{X}^{2} \nonumber \\
&\hspace{0.1in}\leq\delta+LC_{F}^{2}\|F(x_{0})-y\|_{Y}^{1+\varepsilon} \nonumber \\
&\hspace{0.1in}\leq\delta+LC_{F}^{2}\|F(x_{0})-y\|_{Y}^{\varepsilon}(\|F(x_{0})-y^{\delta}\|_{Y}+\delta).
\end{align}
By the mean value inequality, \eqref{F'-bounded}, \eqref{x0} and \eqref{rho;noisy}, we have
\begin{equation} \label{prep;unique;alpha;noisy;2}
\|F(x_{0})-y\|_{Y}\leq\widehat{L}\|x_{0}-x^{\dagger}\|_{X}\leq\sqrt{2\rho}\widehat{L}=\left(\frac{q}{4LC_{F}^{2}}\right)^{\frac{1}{\varepsilon}}. 
\end{equation}
Substituting \eqref{prep;unique;alpha;noisy;2} into \eqref{prep;unique;alpha;noisy;1}, we obtain
\begin{align} \label{prep;unique;alpha;noisy;3}
\|y^{\delta}-F(x_{0})-F'(x_{0})(x^{\dagger}-x_{0})\|_{Y}&\leq\delta+\frac{q}{4}(\|F(x_{0})-y^{\delta}\|_{Y}+\delta) \nonumber \\
&=\left(1+\frac{q}{4}\right)\delta+\frac{q}{4}\|F(x_{0})-y^{\delta}\|_{Y}.
\end{align}
Combining \eqref{DiscrepancyPrinciple}, \eqref{prep;unique;alpha;noisy;3} and \eqref{R;noisy}, we have
\begin{align} \label{unique;alpha;0;noisy}
\|y^{\delta}-F(x_{0})-F'(x_{0})(x^{\dagger}-x_{0})\|_{Y}&<\left(\left(1+\frac{q}{4}\right)\frac{1}{\tau}+\frac{q}{4}\right)\|F(x_{0})-y^{\delta}\|_{Y} \nonumber \\
&=q(1-R)\|F(x_{0})-y^{\delta}\|_{Y}.
\end{align}
It follows from \eqref{R;noisy} that $0<R<\frac{3}{4}$, hence \eqref{unique;alpha} holds with $\omega=\frac{1}{1-R}>1$.

Let us assume that $\alpha_{k}$ is unique, and we will provide a proof of this property later. Analyzing the case where we know only noisy data in the same way as getting \eqref{k+1} and using the notation
\[
\gamma_{k}^{\delta}=\frac{1}{2}\|x_{k}^{\delta}-x^{\dagger}\|_{X}^{2}
\]
yield
\begin{align} \label{prep;k+1;noisy}
\gamma_{k+1}^{\delta}-\gamma_{k}^{\delta}&\leq-\frac{1}{2}\|F'(x_{k}^{\delta})^{\ast}(F'(x_{k}^{\delta})F'(x_{k}^{\delta})^{\ast}+\alpha_{k}I)^{-1}(y^{\delta}-F(x_{k}^{\delta}))\|_{X}^{2} \nonumber \\
&\hspace{0.2in}-\frac{q^{2}}{\alpha_{k}}\|y^{\delta}-F(x_{k}^{\delta})\|_{Y}^{2} \nonumber \\
&\hspace{0.2in}+\frac{q}{\alpha_{k}}\|y^{\delta}-F(x_{k}^{\delta})\|_{Y}\|y^{\delta}-F(x_{k}^{\delta})-F'(x_{k}^{\delta})(x^{\dagger}-x_{k}^{\delta})\|_{Y}.
\end{align}
We now prove
\begin{equation} \label{induction;noisy}
\gamma_{k}^{\delta}\leq\rho\hspace{0.1in}\mbox{for }0\leq k<k_{\ast}
\end{equation}
and
\begin{equation} \label{unique;alpha;k;noisy}
\|y^{\delta}-F(x_{k}^{\delta})-F'(x_{k}^{\delta})(x^{\dagger}-x_{k}^{\delta})\|_{Y}\leq q(1-R)\|F(x_{k}^{\delta})-y^{\delta}\|_{Y}\hspace{0.1in}\mbox{for }0\leq k<k_{\ast}
\end{equation}
by induction. We remark that the estimate \eqref{unique;alpha;k;noisy} is \eqref{unique;alpha} with $\omega=\frac{1}{1-R}$. Hence, we can define $\alpha_{k}$ and $x_{k+1}$ by \eqref{MDP} and \eqref{LM}, respectively. First, \eqref{induction;noisy} and \eqref{unique;alpha;k;noisy} hold for $k=0$ by \eqref{x0} and \eqref{unique;alpha;0;noisy}. Now, we assume \eqref{induction;noisy} and \eqref{unique;alpha;k;noisy} for $k=m<k_{\ast}-1$. In particular, we can define $\alpha_{m}$ and $x_{m+1}$ by \eqref{MDP} and \eqref{LM}, respectively. It follows from \eqref{F'-Lip}, \cite[Lemma\,A.63]{Kirsch}, \eqref{F;inv-Holder} and \eqref{noise} that
\begin{align} \label{prep;FTftFD;noisy}
&\|y-F(x_{m}^{\delta})-F'(x_{m}^{\delta})(x^{\dagger}-x_{m}^{\delta})\|_{Y} \nonumber \\
&\hspace{0.1in}\leq\frac{L}{2}\|x_{m}^{\delta}-x^{\dagger}\|_{X}^{2} \nonumber \\
&\hspace{0.1in}\leq LC_{F}^{2}\|F(x_{m}^{\delta})-y\|_{Y}^{1+\varepsilon} \nonumber \\
&\hspace{0.1in}\leq LC_{F}^{2}\|F(x_{m}^{\delta})-y\|_{Y}^{\varepsilon}(\delta+\|F(x_{m}^{\delta})-y^{\delta}\|_{Y}).
\end{align}
Applying the mean value inequality, \eqref{F'-bounded}, the assumption of induction and \eqref{rho;noisy} yields
\begin{equation} \label{MVI;noisy}
\|F(x_{m}^{\delta})-y\|_{Y}\leq\widehat{L}\|x_{m}^{\delta}-x^{\dagger}\|_{X}\leq\sqrt{2\rho}\widehat{L}=\left(\frac{q}{4LC_{F}^{2}}\right)^{\frac{1}{\varepsilon}}.
\end{equation}
Combining \eqref{noise}, \eqref{prep;FTftFD;noisy} and \eqref{MVI;noisy} implies
\begin{align} \label{unique;alpha;m;noisy}
&\|y^{\delta}-F(x_{m}^{\delta})-F'(x_{m}^{\delta})(x^{\dagger}-x_{m}^{\delta})\|_{Y} \nonumber \\
&\hspace{0.1in}\leq\delta+LC_{F}^{2}\|F(x_{m}^{\delta})-y\|_{Y}^{\varepsilon}(\|F(x_{m}^{\delta})-y^{\delta}\|_{Y}+\delta) \nonumber \\
&\hspace{0.1in}\leq\delta+\frac{q}{4}(\|F(x_{m}^{\delta})-y^{\delta}\|_{Y}+\delta) \nonumber \\
&\hspace{0.1in}=\left(1+\frac{q}{4}\right)\delta+\frac{q}{4}\|F(x_{m}^{\delta})-y^{\delta}\|_{Y}.
\end{align}
Substituting \eqref{unique;alpha;m;noisy} into \eqref{prep;k+1;noisy} yields
\begin{align} \label{prep;m+1;3;noisy}
\gamma_{m+1}^{\delta}-\gamma_{m}^{\delta}&\leq-\frac{q^{2}}{\alpha_{m}}\|y^{\delta}-F(x_{m}^{\delta})\|_{Y}^{2} \nonumber \\
&\hspace{0.1in}+\frac{q}{\alpha_{m}}\|y^{\delta}-F(x_{m}^{\delta})\|_{Y}\left(\left(1+\frac{q}{4}\right)\delta+\frac{q}{4}\|F(x_{m}^{\delta})-y^{\delta}\|_{Y}\right) \nonumber \\
&=-\frac{3q^{2}}{4\alpha_{m}}\|y^{\delta}-F(x_{m}^{\delta})\|_{Y}^{2}+\frac{q}{\alpha_{m}}\left(1+\frac{q}{4}\right)\delta\|y^{\delta}-F(x_{m}^{\delta})\|_{Y}.
\end{align}
Combining \eqref{prep;m+1;3;noisy}, \eqref{DiscrepancyPrinciple}, \eqref{R;noisy} and \eqref{alpha_k;L} yields
\begin{align} \label{m+1;2;noisy}
\gamma_{m+1}^{\delta}-\gamma_{m}^{\delta}&<-\frac{3q^{2}}{4\alpha_{m}}\|y^{\delta}-F(x_{m}^{\delta})\|_{Y}^{2}+\frac{q}{\alpha_{m}}\left(1+\frac{q}{4}\right)\frac{1}{\tau}\|y^{\delta}-F(x_{m}^{\delta})\|_{Y}^{2} \nonumber \\
&=-\frac{q^{2}}{\alpha_{m}}R\|y^{\delta}-F(x_{m}^{\delta})\|_{Y}^{2}\leq-\frac{q(1-q)}{\widehat{L}^{2}}R\|y^{\delta}-F(x_{m}^{\delta})\|_{Y}^{2}
\end{align}
It follows from \eqref{m+1;2;noisy} and \eqref{R;noisy} that
\begin{equation} \label{MD;noisy}
\gamma_{m+1}^{\delta}\leq\gamma_{m}^{\delta}\leq\rho.
\end{equation}
Next, in the same way as getting \eqref{prep;FTftFD;noisy}, we obtain
\begin{align} \label{prep;FTftFD;m+1;noisy}
&\|y-F(x_{m+1}^{\delta})-F'(x_{m+1}^{\delta})(x^{\dagger}-x_{m+1}^{\delta})\|_{Y} \nonumber \\
&\hspace{0.1in}\leq LC_{F}^{2}\|F(x_{m+1}^{\delta})-y\|_{Y}^{\varepsilon}(\delta+\|F(x_{m+1}^{\delta})-y^{\delta}\|_{Y}).
\end{align}
Similarly to \eqref{MVI;noisy}, we have
\begin{equation} \label{MVI;m+1;noisy}
\|F(x_{m+1}^{\delta})-y\|_{Y}\leq\left(\frac{q}{4LC_{F}^{2}}\right)^{\frac{1}{\varepsilon}}.
\end{equation}
Combining \eqref{noise}, \eqref{prep;FTftFD;m+1;noisy}, \eqref{MVI;m+1;noisy}, \eqref{DiscrepancyPrinciple} and \eqref{R;noisy} yields
\begin{align*}
& \|y^{\delta}-F(x_{m+1}^{\delta})-F'(x_{m+1}^{\delta})(x^{\dagger}-x_{m+1}^{\delta})\|_{Y} \\
&\hspace{0.1in}\leq\delta+LC_{F}^{2}\|F(x_{m+1}^{\delta})-y\|_{Y}^{\varepsilon}(\delta+\|F(x_{m+1}^{\delta})-y^{\delta}\|_{Y}) \\
&\hspace{0.1in}\leq\left(1+\frac{q}{4}\right)\delta+\frac{q}{4}\|F(x_{m+1}^{\delta})-y^{\delta}\|_{Y} \\
&\hspace{0.1in}\leq\left(1+\frac{q}{4}\right)\frac{1}{\tau}\|F(x_{m+1}^{\delta})-y^{\delta}\|_{Y}+\frac{q}{4}\|F(x_{m+1}^{\delta})-y^{\delta}\|_{Y} \\
&\hspace{0.1in}=q(1-R)\|F(x_{m+1}^{\delta})-y^{\delta}\|_{Y},
\end{align*}
which is \eqref{unique;alpha;k;noisy} for $k=m+1$. Therefore, by induction, it can be shown that \eqref{induction;noisy} holds and we can define $\alpha_{k}$ and $x_{k+1}$ for $0\leq k<k_{\ast}$. 

We next show $k_{\ast}<\infty$. It follows from \eqref{m+1;2;noisy} and \eqref{x0} that
\begin{align} \label{k;ast}
\frac{q(1-q)}{\widehat{L}^{2}}R\sum_{k=0}^{k_{\ast}-1}\|y^{\delta}-F(x_{k}^{\delta})\|_{Y}^{2}&\leq\frac{1}{2}\|x_{0}-x^{\dagger}\|_{X}^{2}-\frac{1}{2}\|x_{k_{\ast}}^{\delta}-x^{\dagger}\|_{X}^{2} \nonumber \\
&\leq\frac{1}{2}\|x_{0}-x^{\dagger}\|_{X}^{2}\leq\rho.
\end{align}
We remark that we have
\begin{equation} \label{prep;k;ast}
\sum_{k=0}^{k_{\ast}-1}\|y^{\delta}-F(x_{k}^{\delta})\|_{Y}^{2}>(\tau\delta)^{2}k_{\ast}
\end{equation}
from \eqref{DiscrepancyPrinciple}.
Hence, combining \eqref{k;ast} and \eqref{prep;k;ast}, we obtain
\[
k_{\ast}\leq\frac{\widehat{L}^{2}\rho}{q(1-q)R(\tau\delta)^{2}}<\infty.
\]

Finally, we establish a convergence rate. Using \eqref{k;ast}, \eqref{prep;k;ast} and \eqref{asm;noisy} yields
\begin{align*}
\frac{1}{2}\|x_{k_{\ast}}^{\delta}-x^{\dagger}\|_{X}^{2}&\leq\frac{1}{2}\|x_{0}-x^{\dagger}\|_{X}^{2}-\frac{q(1-q)}{\widehat{L}^{2}}R(\tau\delta)^{2}k_{\ast} \\
&\leq\rho-\frac{q(1-q)}{\widehat{L}^{2}}R(\tau\delta)^{2}k_{\ast} \\
&\leq C\delta^{2}-\frac{q(1-q)}{\widehat{L}^{2}}R(\tau\delta)^{2}k_{\ast}=C'\delta^{2}.
\end{align*}
Hence, we have completed the proof.
\end{proof}

With an additional assumption, we can have a logarithmic estimate for $k_{\ast}$. The details are as follows.
\begin{proposition} \label{prop:kast}
Let the assumptions of Theorem~\ref{thm:noisy;Hilbert} hold. In addition, we assume
\begin{equation} \label{nu;additional}
0<q<\frac{2\sqrt{2}\widehat{L}C_{F}}{1+2\sqrt{2}\widehat{L}C_{F}}
\end{equation}
and $\varepsilon=1$. Then, we have
\[
k_{\ast}(\delta, y^{\delta})=O(1+|\log\delta|).
\]
\end{proposition}
\begin{proof}
To establish the logarithmic estimate for $k_{\ast}$, we prove that
\begin{equation} \label{aim;kast}
\|y^{\delta}-F(x_{k+1}^{\delta})\|_{Y}\leq\widetilde{q}\|y^{\delta}-F(x_{k}^{\delta})\|_{Y},\hspace{0.2in}0\leq k<k_{\ast},
\end{equation}
where $\widetilde{q}$ is a number to be chosen later.

It follows from \eqref{MDP'} that
\begin{align} \label{prep;kast;1}
&q\|y^{\delta}-F(x_{k}^{\delta})\|_{Y} \nonumber \\
&\hspace{0.1in}\geq\|y^{\delta}-F(x_{k+1}^{\delta})\|_{Y}-\|F(x_{k+1}^{\delta})-F(x_{k}^{\delta})-F'(x_{k}^{\delta})(x_{k+1}^{\delta}-x_{k}^{\delta}))\|_{Y}.
\end{align}
Using \cite[Lemma\,A.63]{Kirsch} and \eqref{F;inv-Holder}, we have
\begin{align} \label{prep;kast;2}
&\|F(x_{k+1}^{\delta})-F(x_{k}^{\delta})-F'(x_{k}^{\delta})(x_{k+1}^{\delta}-x_{k}^{\delta}))\|_{Y} \nonumber \\
&\hspace{0.1in}\leq\frac{L}{2}\|x_{k+1}^{\delta}-x_{k}^{\delta}\|_{X}^{2} \nonumber \\
&\hspace{0.1in}\leq\frac{LC_{F}}{\sqrt{2}}\|x_{k+1}^{\delta}-x_{k}^{\delta}\|_{X}\|F(x_{k+1}^{\delta})-F(x_{k}^{\delta})\|_{Y}.
\end{align}
Substituting \eqref{prep;kast;2} into \eqref{prep;kast;1}, we obtain
\begin{align} \label{prep;kast;3}
&q\|y^{\delta}-F(x_{k}^{\delta})\|_{Y} \nonumber \\
&\hspace{0.1in}\geq\|y^{\delta}-F(x_{k+1}^{\delta})\|_{Y}-\frac{LC_{F}}{\sqrt{2}}\|x_{k+1}^{\delta}-x_{k}^{\delta}\|_{X}\|F(x_{k+1}^{\delta})-F(x_{k}^{\delta})\|_{Y} \nonumber \\
&\hspace{0.1in}\geq\|y^{\delta}-F(x_{k+1}^{\delta})\|_{Y} \nonumber \\
&\hspace{0.3in}-\frac{LC_{F}}{\sqrt{2}}\|x_{k+1}^{\delta}-x_{k}^{\delta}\|_{X}(\|y^{\delta}-F(x_{k+1}^{\delta})\|_{Y}+\|y^{\delta}-F(x_{k}^{\delta})\|_{Y}).
\end{align}
It follows from \eqref{KNS;4.5}, \eqref{mono;noisy}, \eqref{x0}, \eqref{rho;noisy} and \eqref{nu;additional} that
\begin{align} \label{prep;kast;4}
1-\frac{LC_{F}}{\sqrt{2}}\|x_{k+1}^{\delta}-x_{k}^{\delta}\|_{X}&\geq 1-\frac{LC_{F}}{\sqrt{2}}\|x_{0}-x^{\dagger}\|_{X}\geq 1-LC_{F}\sqrt{\rho} \nonumber \\
&=1-\frac{q}{4\sqrt{2}\widehat{L}C_{F}}>0.
\end{align}
Combining \eqref{prep;kast;3} and \eqref{prep;kast;4} yields
\begin{equation} \label{prep;kast;5}
\|y^{\delta}-F(x_{k+1}^{\delta})\|_{Y}\leq\frac{q+\frac{LC_{F}}{\sqrt{2}}\|x_{k+1}^{\delta}-x_{k}^{\delta}\|_{X}}{1-\frac{LC_{F}}{\sqrt{2}}\|x_{k+1}^{\delta}-x_{k}^{\delta}\|_{X}}\|y^{\delta}-F(x_{k}^{\delta})\|_{Y}.
\end{equation}
Moreover, we obtain
\begin{equation} \label{prep;kast;6}
\frac{q+\frac{LC_{F}}{\sqrt{2}}\|x_{k+1}^{\delta}-x_{k}^{\delta}\|_{X}}{1-\frac{LC_{F}}{\sqrt{2}}\|x_{k+1}^{\delta}-x_{k}^{\delta}\|_{X}}\leq\frac{q+\frac{LC_{F}}{\sqrt{2}}\|x_{0}-x^{\dagger}\|_{X}}{1-\frac{LC_{F}}{\sqrt{2}}\|x_{0}-x^{\dagger}\|_{X}}=:\widetilde{q}
\end{equation}
since $y=\frac{q+\frac{LC_{F}}{\sqrt{2}}s}{1-\frac{LC_{F}}{\sqrt{2}}s}$ is monotonically increasing for $s\leq\frac{\sqrt{2}}{LC_{F}}$ and we have
\[
\|x_{k+1}^{\delta}-x_{k}^{\delta}\|_{X}\leq\|x_{0}-x^{\dagger}\|_{X}\leq\sqrt{2\rho}=\frac{q}{4\widehat{L}LC_{F}^{2}}<\frac{\sqrt{2}}{LC_{F}}
\]
by \eqref{KNS;4.5}, \eqref{mono;noisy}, \eqref{x0}, \eqref{rho;noisy} and \eqref{nu;additional}. Combining \eqref{prep;kast;5} and \eqref{prep;kast;6} yields \eqref{aim;kast}. Hence, it follows from \eqref{DiscrepancyPrinciple} and repeating \eqref{aim;kast} that
\[
\tau\delta<\|y^{\delta}-F(x_{k_{\ast}-1}^{\delta})\|_{Y}\leq\widetilde{q}\|y^{\delta}-F(x_{k_{\ast}-2}^{\delta})\|_{Y}\leq\widetilde{q}^{k_{\ast}-1}\|y^{\delta}-F(x_{0})\|_{Y}.
\]
By \eqref{x0}, \eqref{rho;noisy} and \eqref{nu;additional}, we have $\widetilde{q}<1$, which yields the desired estimate for $k_{\ast}$.
\end{proof}

\section{Application to IFM} \label{section:application}
In this section, we present global reconstruction algorithms for approximate solutions to the IFM by using the LM-method. We first note that Theorem~\ref{thm:exact;Hilbert} and Theorem~\ref{thm:noisy;Hilbert} hold provided that an appropriate initial data, namely $x_{0}$ as given in \eqref{x0}. This means that these convergences are local. These local convergences, combined with the discussion on how such initial data can be obtained (see Lemma~\ref{lem:finite}), lead to global algorithms. By combining the local convergences and Lemma~\ref{lem:finite} with the stability estimate established in \cite{AS}, we design global algorithms to address the IFM for exact and noisy data, respectively.

Let us give a list of assumptions throughout this section.
\begin{itemize}
\setlength{\leftskip}{-0.1in}
\setlength{\itemsep}{-3.5pt}
\item[-] $X$ and $Y$ be Banach spaces;
\item[-] $A\subset X$ be an open set;
\item[-] $W\subset X$ be a finite-dimensional subspace;
\item[-] $F\in C^{1}(A, Y)$ be such that $\left.F\right|_{W\cap A}$ and $\left(\left.F\right|_{W\cap A}\right)'$ are Lipschitz continuous; 
\item[-] $Q:Y\rightarrow Y$ be a continuous finite-rank operator;
\item[-] $K\subseteq W\cap A$ be a compact set;
\item[-] and $\widetilde{C}$ be a positive constant such that
\begin{equation} \label{key}
\|x-\widetilde{x}\|_{X}\leq 2\widetilde{C}\|Q(F(x))-Q(F(\widetilde{x}))\|_{Y},\hspace{0.2in}x, \widetilde{x}\in W\cap A.
\end{equation}
\end{itemize}

We now consider the equation~$y=Q(F(x^{\dagger}))\in Y$, where $x^{\dagger}\in K$ is the unknown data. Our aim is to reconstruct the solution~$x^{\dagger}$ from exact data~$y$ and noisy data~$y^{\delta}$ by applying Theorem~\ref{thm:exact;Hilbert} and Theorem~\ref{thm:noisy;Hilbert}, respectively. We remark that we can introduce inner products for the domain $W$ and the range of $Q$ since these spaces are finite-dimensional. Moreover, the operator~$Q\circ F: W\cap A\rightarrow Y$ is Fr\'{e}chet derivative and its derivative $(Q\circ F)'$ is Lipschitz continuous and bounded in $K$. It follows from these properties and the assumption~\eqref{key} that the operator $Q\circ F$ satisfies Assumption~\ref{asm;Holder}. Then, we can apply Theorem~\ref{thm:exact;Hilbert} and Theorem~\ref{thm:noisy;Hilbert} if the constants $q$ and $\rho$ satisfy the assumptions in each theorem, and the initial guess~$x_{0}$ satisfies the condition~\eqref{x0}. On the other hand, we can find the initial guess~$x_{0}$ satisfying~\eqref{x0} as stated in \cite[Section~3.2]{AS}. Indeed, by the compactness of $K$, there exists a finite lattice $\{x^{(j)}\mid j\in J\}\subseteq K$ such that
\begin{equation} \label{fromCompact}
K\subseteq\bigcup_{j\in J}B_{X}\left(x^{(j)}, \frac{\rho}{2\widetilde{L}\widetilde{C}\|Q\|_{\mathcal{L}(Y)}}\right),
\end{equation}
where $\rho>0$ is given by \eqref{rho} for exact data and \eqref{rho;noisy} for noisy data, $\widetilde{C}$ is given by \eqref{key} and $\widetilde{L}>0$ is a Lipschitz constant of $F$, that is
\begin{equation} \label{5;F-Lip}
\|F(x)-F(\widetilde{x})\|_{Y}\leq\widetilde{L}\|x-\widetilde{x}\|_{X},\hspace{0.2in}x, \widetilde{x}\in W\cap A.
\end{equation}
To find the initial guess~$x_{0}$, we first introduce the following lemma, and then take $x^{(j)}$ in Lemma~\ref{lem:finite} as $x_{0}$.

\begin{lemma}[\textup{\cite[Lemma\,2]{AS}}] \label{lem:finite}
Under the above assumptions, we have \vspace{-0.05in}
\begin{enumerate}
\renewcommand{\labelenumi}{(\roman{enumi})}
\item there exists $j\in J$ such that
\begin{equation} \label{QF}
\|Q(F(x^{(j)}))-Q(F(x^{\dagger}))\|_{Y}<\frac{\rho}{2\widetilde{C}};
\end{equation}
\item if \eqref{QF} holds for some $j\in J$, then we have
\[
\|x^{(j)}-x^{\dagger}\|_{X}<\rho.
\]
\end{enumerate}
\end{lemma}
\begin{proof}
\begin{enumerate}
\renewcommand{\labelenumi}{(\roman{enumi})}
\item It follows from the compactness of $K$, $x^{\dagger}\in K$ and \eqref{fromCompact} that there exists $j\in J$ such that
\begin{equation} \label{prep;i;1}
\|x^{(j)}-x^{\dagger}\|_{X}<\frac{\rho}{2\widetilde{L}\widetilde{C}\|Q\|_{\mathcal{L}(Y)}}.
\end{equation}
Moreover, using $x^{(j)}\in K$ and \eqref{5;F-Lip}, we have
\begin{equation} \label{prep;i;2}
\|F(x^{(j)})-F(x^{\dagger})\|_{Y}\leq\widetilde{L}\|x^{(j)}-x^{\dagger}\|_{X}.
\end{equation}
Combining \eqref{prep;i;1} and \eqref{prep;i;2} yields
\begin{align*} 
\|Q(F(x^{(j)}))-Q(F(x^{\dagger}))\|_{Y}&\leq\|Q\|_{\mathcal{L}(Y)}\|F(x^{(j)})-F(x^{\dagger})\|_{Y} \\
&\leq\|Q\|_{\mathcal{L}(Y)}\widetilde{L}\|x^{(j)}-x^{\dagger}\|_{X}\leq\frac{\rho}{2\widetilde{C}}.
\end{align*}
\item Assume \eqref{QF} for some $j\in J$. It follows from \eqref{key} that
\begin{equation} \label{prep;ii}
\|x^{(j)}-x^{\dagger}\|_{X}\leq 2\widetilde{C}\|Q(F(x^{(j)}))-Q(F(x^{\dagger}))\|_{Y}.
\end{equation}
Combining \eqref{prep;ii} with \eqref{QF} yields
\[
\|x^{(j)}-x^{\dagger}\|_{X}<2\widetilde{C}\frac{\rho}{2\widetilde{C}}=\rho.
\]
\end{enumerate}
The proof is completed.
\end{proof}

We now present the following global reconstruction algorithm for approximate solutions to the IFM, starting with the one for exact data. We emphasize that, according to Theorem~\ref{thm:exact;Hilbert}, the sequence $\{x_k\}_{k=0}^{\infty}$ generated by Algorithm~\ref{al:exact} is guaranteed to converge to $x^{\dagger}$. Furthermore, as \eqref{key} is assumed, note that $\alpha_{k}$ in the following algorithm is unique. Since we have already established the convergence rate in Theorem~\ref{thm:exact;Hilbert}, we are able to calculate the number of iterations required to achieve the desired level of accuracy. We define this number of iterations as $M$.
\begin{lstlisting}[caption=Reconstruction of $x^{\dagger}$ from $Q(F(x^{\dagger}))$ \label{al:exact}]
1: Input $X, Y, W, K, Q, F, Q(F(x^{\dagger})), \rho, \widetilde{L}, \widetilde{C}$ and $M$.
2: Equip $W$ and $Q(Y)$ with equivalent euclidean scalar products. 
3: Find a finite lattice $\{x^{(j)}\mid j\in J\}\subseteq K$ so that $\eqref{fromCompact}$ is satisfied. 
4: for $j\in J$ do 
5:   Compute $Q(F(x^{(j)}))$. 
6:   if $\eqref{QF}$ is satisfied then 
7:     Set $x_{0}=x^{(j)}$. 
8:     exit for
9:   end if 
10: end for 
11: for $k=0, \ldots, M$ do
12:   Set $\alpha_{k}$ satisfying
            $\alpha_{k}\|Q\left(\left.F\right|_{W\cap A}\right)'(x_{k})\left(\left.F\right|_{W\cap A}\right)'(x_{k})^{\ast}Q^{\ast}+\alpha_{k}I)^{-1}(y-Q(F(x_{k})))\|_{Y}$
                $=q\|y-Q(F(x_{k}))\|_{Y}$.
13:   Set $x_{k+1}=x_{k}+\left(\left(\left(\left.F\right|_{W\cap A}\right)'(x_{k})^{\ast}Q^{\ast}Q\left(\left.F\right|_{W\cap A}\right)'(x_{k})\right)+\alpha_{k}I\right)^{-1}$                                             $\left(\left.F\right|_{W\cap A}\right)'(x_{k})^{\ast}Q^{\ast}(y-Q(F(x_{k})))$.
14:   if the stopping criterion is satisfied then
15:     exit for
16:   end if
17: end for
18: Output $x_{k+1}$.
\end{lstlisting}

Next, we present a global reconstruction algorithm for approximate solutions to the IFM for the case where only noisy data $y^{\delta}$ satisfying \eqref{noise} is available. As stated in Theorem~\ref{thm:noisy;Hilbert}, the sequence $\{x_k^{\delta}\}$ decreases monotonically up to $k=k_{\ast}$. Additionally, given \eqref{key}, it follows that $\alpha_{k}$ in the following algorithm is unique.
\begin{lstlisting}[caption=Reconstruction of $x^{\dagger}$ from $y^{\delta}$ \label{al:noisy}]
1: Input $X, Y, W, K, Q, F, \delta, y^{\delta}, \rho, \widetilde{L}, \widetilde{C}$, $\tau$ and $M$.
2: Equip $W$ and $Q(Y)$ with equivalent euclidean scalar products. 
3: Find a finite lattice $\{x^{(j)}\mid j\in J\}\subseteq K$ so that $\eqref{fromCompact}$ is satisfied. 
4: for $j\in J$ do 
5:   Compute $Q(F(x^{(j)}))$. 
6:   if $\eqref{QF}$ is satisfied then 
7:     Set $x_{0}=x^{(j)}$. 
8:     exit for
9:   end if 
10: end for 
11: for $k=0, \ldots, M$ do
12:   Set $\alpha_{k}$ satisfying
            $\alpha_{k}\|Q\left(\left.F\right|_{W\cap A}\right)'(x_{k}^{\delta})\left(\left.F\right|_{W\cap A}\right)'(x_{k}^{\delta})^{\ast}Q^{\ast}+\alpha_{k}I)^{-1}(y^{\delta}-Q(F(x_{k}^{\delta})))\|_{Y}$
                $=q\|y^{\delta}-Q(F(x_{k}^{\delta}))\|_{Y}$.
13:   Set $x_{k+1}^{\delta}=x_{k}^{\delta}+\left(\left(\left(\left.F\right|_{W\cap A}\right)'(x_{k}^{\delta})^{\ast}Q^{\ast}Q\left(\left.F\right|_{W\cap A}\right)'(x_{k}^{\delta})\right)+\alpha_{k}I\right)^{-1}$                                             $\left(\left.F\right|_{W\cap A}\right)'(x_{k}^{\delta})^{\ast}Q^{\ast}(y^{\delta}-Q(F(x_{k}^{\delta})))$.
14:   if $\|y^{\delta}-Q(F(x_{k+1}^{\delta}))\|_{Y}\leq\tau\delta$ is satisfied then
15:      exit for
16:   end if
17:   else if the stopping criterion is satisfied then
18:      exit for
19:   end if
20: end for
21: Output $x_{k+1}^{\delta}$.
\end{lstlisting}

\begin{remark}
In addition to the stopping criterion, we have to consider the discrepancy principle; we should stop the iteration after $k_{\ast}$ steps, where 
\[
k_{\ast}=\min\{k\mid \|y^{\delta}-Q(F(x_{k+1}^{\delta}))\|_{Y}\leq\tau\delta\}.
\]
If the iterate $x_{k+1}^{\delta}$ satisfies $\|y^{\delta}-Q(F(x_{k+1}^{\delta}))\|_{Y}\leq\tau\delta$ in Algorithm~\ref{al:noisy}, then the number $k+1$ is equal to $k_{\ast}$. Hence, it follows from Theorem~\ref{thm:noisy;Hilbert} that the output~$x_{k+1}^{\delta}$ satisfies \eqref{stab;noisy} if \eqref{asm;noisy} holds for some $C>0$.

There is a possibility that the iteration may stop before $k_{\ast}$ steps. Then, using arguments similar to the derivation of the convergence rate~\eqref{stab;noisy}, we have
\[
\frac{1}{2}\|x_{M+1}^{\delta}-x^{\dagger}\|_{X}^{2}\leq\left(C-\frac{q(1-q)}{\widetilde{L}^{2}}R\tau^{2}(M+1)\right)\delta^{2}.
\]
\end{remark}

\section{Summary and discussion}
Before closing this paper, we give a summary and discussion for our results. Let us start giving a summary. In this paper, we study convergence of the LM-method assuming H\"{o}lder stability estimate. We prove local convergence of the LM-method and establish convergence rates for exact and noisy data, respectively. Furthermore, we apply these results to propose global reconstruction algorithms for approximate solutions to the IFM. In a global reconstruction algorithm for exact data, based on convergence rate we have established, we can determine the number of iterations required to achieve the desired level of accuracy. For noisy data, it may be necessary to stop the algorithm using the discrepancy principle; however, in this case, the number of iterations needed to achieve the desired accuracy can also be determined.

We have the following discussion of our results. As mentioned in Section~\ref{section:introduction}, we had an expectation based on the numerical study by \cite{Hohage} that the LM-method gives faster convergence compared to the Landweber iteration. However, the convergence rates obtained in this paper seem to be similar to those in \cite{dHQS} and \cite{MG}. Nevertheless, we anticipate that the LM-method converges faster than the Landweber iteration in most cases. The reason for this is that we evaluated $-\frac{1}{2}\|F'(x_{k})^{\ast}(F'(x_{k})F'(x_{k})^{\ast}+\alpha_{k}I)^{-1}(y-F(x_{k}))\|_{X}^{2}\leq0$ in \eqref{m+1;prep;2} in the proof of Theorem~\ref{thm:exact;Hilbert}, but this term may be nonzero, that is, negative in most cases. Hence, it would be important to compare the constants in our theorem with those in \cite{dHQS} and \cite{MG} through numerical experiments. Furthermore, in order to widen the application, it is important to extend our results to Banach spaces.

\bigskip
\textbf{Acknowlegements.} The first author was supported in part by JST SPRING, Grant Number JPMJSP2125. The third author was partially supported by JSPS KAKENHI (Grant Nos.\ JP22K03366).

\bibliographystyle{plain}
\bibliography{002}

\begin{thebibliography}{10}

\bibitem{AS}
Giovanni~S. Alberti and Matteo Santacesaria.
\newblock Infinite-dimensional inverse problems with finite measurements.
\newblock {\em Arch. Ration. Mech. Anal.}, 243(1):1--31, 2022.

\bibitem{Alessandrini}
Giovanni Alessandrini.
\newblock Stable determination of conductivity by boundary measurements.
\newblock {\em Appl. Anal.}, 27(1-3):153--172, 1988.

\bibitem{AHG}
Giovanni Alessandrini, Maarten~V. de~Hoop, and Romina Gaburro.
\newblock Uniqueness for the electrostatic inverse boundary value problem with
  piecewise constant anisotropic conductivities.
\newblock {\em Inverse Problems}, 33(12):125013, 24, 2017.

\bibitem{AHGS17}
Giovanni Alessandrini, Maarten~V. de~Hoop, Romina Gaburro, and Eva Sincich.
\newblock Lipschitz stability for the electrostatic inverse boundary value
  problem with piecewise linear conductivities.
\newblock {\em J. Math. Pures Appl. (9)}, 107(5):638--664, 2017.

\bibitem{AV}
Giovanni Alessandrini and Sergio Vessella.
\newblock Lipschitz stability for the inverse conductivity problem.
\newblock {\em Adv. in Appl. Math.}, 35(2):207--241, 2005.

\bibitem{BHQ}
Elena Beretta, Maarten~V. de~Hoop, and Lingyun Qiu.
\newblock Lipschitz stability of an inverse boundary value problem for a
  {S}chr\"odinger-type equation.
\newblock {\em SIAM J. Math. Anal.}, 45(2):679--699, 2013.

\bibitem{dHQS}
Maarten~V. de~Hoop, Lingyun Qiu, and Otmar Scherzer.
\newblock Local analysis of inverse problems: {H}\"{o}lder stability and
  iterative reconstruction.
\newblock {\em Inverse Problems}, 28(4):045001, 2012.

\bibitem{Hanke}
Martin Hanke.
\newblock A regularizing {L}evenberg-{M}arquardt scheme, with applications to
  inverse groundwater filtration problems.
\newblock {\em Inverse Problems}, 13(1):79--95, 1997.

\bibitem{HNS}
Martin Hanke, Andreas Neubauer, and Otmar Scherzer.
\newblock A convergence analysis of the {L}andweber iteration for nonlinear
  ill-posed problems.
\newblock {\em Numer. Math.}, 72(1):21--37, 1995.

\bibitem{Harrach}
Bastian Harrach.
\newblock Uniqueness, stability and global convergence for a discrete inverse
  elliptic {R}obin transmission problem.
\newblock {\em Numer. Math.}, 147(1):29--70, 2021.

\bibitem{Hohage}
Thorsten Hohage.
\newblock Iterative regularization methods in inverse scattering.
\newblock In {\em Proceedings of Inverse Problems in Engineering: Theory and
  Practice 3rd Int. Conference on Inverse Problems in Engineering, Port Ludlow,
  WA, June 13--18}, 1999.

\bibitem{KNS}
Barbara Kaltenbacher, Andreas Neubauer, and Otmar Scherzer.
\newblock {\em Iterative regularization methods for nonlinear ill-posed
  problems}, volume~6 of {\em Radon Series on Computational and Applied
  Mathematics}.
\newblock Walter de Gruyter GmbH \& Co. KG, Berlin, 2008.

\bibitem{KSS}
Barbara Kaltenbacher, Frank Sch\"opfer, and Thomas Schuster.
\newblock Iterative methods for nonlinear ill-posed problems in {B}anach
  spaces: convergence and applications to parameter identification problems.
\newblock {\em Inverse Problems}, 25(6):065003, 19, 2009.

\bibitem{Kirsch}
Andreas Kirsch.
\newblock {\em An introduction to the mathematical theory of inverse problems},
  volume 120 of {\em Applied Mathematical Sciences}.
\newblock Springer, Cham, third edition, 2021.

\bibitem{MG}
Gaurav Mittal and Ankik~Kumar Giri.
\newblock Improved local convergence analysis of the {L}andweber iteration in
  {B}anach spaces.
\newblock {\em Arch. Math. (Basel)}, 120(2):195--202, 2023.

\bibitem{SLS}
F.~Sch\"opfer, A.~K. Louis, and T.~Schuster.
\newblock Nonlinear iterative methods for linear ill-posed problems in {B}anach
  spaces.
\newblock {\em Inverse Problems}, 22(1):311--329, 2006.

\bibitem{SKHK}
Thomas Schuster, Barbara Kaltenbacher, Bernd Hofmann, and Kamil~S. Kazimierski.
\newblock {\em Regularization methods in {B}anach spaces}, volume~10 of {\em
  Radon Series on Computational and Applied Mathematics}.
\newblock Walter de Gruyter GmbH \& Co. KG, Berlin, 2012.

\end{thebibliography}

\bigskip \noindent
Akari Ishida \\
Graduate School of Mathematics, Nagoya University, \\
Furocho, Chikusa-ku, Nagoya, Aichi 464-8602, Japan \\
E-mail: \texttt{akari.ishida.c5@math.nagoya-u.ac.jp} 

\bigskip \noindent
Sei Nagayasu \\
Graduate School of Science, University of Hyogo, \\
2167 Shosha, Himeji, Hyogo 671-2201, Japan \\
E-mail: \texttt{sei@sci.u-hyogo.ac.jp}

\bigskip \noindent
Gen Nakamura \\
Department of Mathematics, Hokkaido University, Sapporo 060-0810,
Research Institute for Electronic Science, Hokkaido University, Sapporo 001-0020, Japan, \\
E-mail: \texttt{nakamuragenn@gmail.com}
\end{document}